\documentclass[11pt]{article}

\usepackage[T1]{fontenc}
\usepackage{lmodern}
\usepackage{amsmath,amssymb,amsthm}
\usepackage[margin=1in]{geometry}
\usepackage[hidelinks]{hyperref}

\hypersetup{
  pdftitle={The Maximum of per(I-A) in Odd Order},
  pdfauthor={Yair Lavi}
}

\newtheorem{theorem}{Theorem}[section]
\newtheorem{proposition}[theorem]{Proposition}
\newtheorem{lemma}[theorem]{Lemma}
\newtheorem{corollary}[theorem]{Corollary}
\theoremstyle{remark}
\newtheorem{remark}[theorem]{Remark}

\title{The Maximum of $\operatorname{per}(I-A)$ in Odd Order}
\author{Yair Lavi}
\date{8 August 2026}

\begin{document}
\maketitle

\begin{abstract}
Let $\Omega_n$ denote the set of $n\times n$ doubly stochastic
matrices. Kim and Roush conjectured in 1981 that, for $n=2k+1>1$,

\[
\max_{A\in\Omega_{2k+1}}\operatorname{per}(I-A)=3\cdot 2^{k-2}.
\]

They proposed the following block construction as an extremizer:

\[
A_\star=\frac12(J_3-I_3)\oplus P_2^{\oplus(k-1)},
\qquad
P_2=\begin{pmatrix}0&1\\1&0\end{pmatrix}.
\]

Here $J_3$ is the $3\times3$ all-ones matrix. They did not claim uniqueness.
We fully prove their conjecture and classify equality: the maximizers are
exactly the simultaneous-permutation conjugates of $A_\star$.
\end{abstract}



\section{Introduction}

For an $n\times n$ matrix $D=(d_{ij})$, its permanent is

\[
\operatorname{per}D=\sum_{\sigma\in S_n}\prod_{i=1}^n d_{i,\sigma(i)}.
\]

We write $\Omega_n$ for the Birkhoff polytope of nonnegative doubly
stochastic $n\times n$ matrices. The problem considered here is to maximize
$\operatorname{per}(I-A)$ over $A\in\Omega_n$. Although $I-A$ generally has
negative off-diagonal entries, stochasticity gives this permanent the positive
functional-digraph expansion proved in Proposition 3.1.

Kim and Roush \cite{kimroush1981} proved that for nonnegative row-stochastic matrices

\[
\max_A\operatorname{per}(I-A)=2^{\lfloor n/2\rfloor}.
\tag{1.1}
\]

Their proof established the functional-digraph identity used below. For even
$n=2m$, the row-stochastic maximum is attained by a fixed-point-free
involution, which is doubly stochastic, and hence the doubly stochastic
maximum is $2^m$. For odd $n=2k+1$, (1.1) gives the upper bound $2^k$.
Kim and Roush exhibited the doubly stochastic matrix

\[
A_\star=\frac12(J_3-I_3)\oplus P_2^{\oplus(k-1)},
\qquad
P_2=\begin{pmatrix}0&1\\1&0\end{pmatrix},
\tag{1.2}
\]

and conjectured that its value $3\cdot2^{k-2}$ is maximal. They proved the
case $n=3$, where the value is $3/2$ \cite[Proposition~5]{kimroush1981}. Thus the open range
in their paper began at $n=5$.

Minc later added the conjecture to his catalogue as Conjecture 35
\cite[p.~133]{minc1987}. In their subsequent survey of Minc's list, Cheon
and Wanless \cite{cheonwanless2005} reported no progress. Troanca
\cite{troanca2008} reproduced the exact construction and stated in 2008 that
no partial results were known to the author. Chen and Cao \cite{chencao2018}
subsequently solved a neighboring problem with
fixed total matrix mass. In their notation, for row- and doubly
substochastic matrices of total mass $s=\sum_{i,j}a_{ij}$, they obtained the
exact common maximum

\[
2^{e/2}\left[1+\left(\frac{s-e}{2}\right)^2\right],
\tag{1.3}
\]

where $e$ is the largest even integer not exceeding $s$, whenever $n$ is
even or $n$ is odd and $s\le n-1$. For odd $n$ and $n-1<s\le n$, they
determined the row-substochastic maximum but left the doubly substochastic
case open \cite[Theorems~3.1, 3.13, and~4.1]{chencao2018}. The endpoint $s=n$ is exactly
the odd doubly stochastic problem. They also gave an independent proof of
the $n=3$ value \cite[Lemma~4.2]{chencao2018}.

Our main result settles the conjecture in full and classifies all equality
cases.

\begin{theorem}[Kim--Roush conjecture]
For every integer $k\ge1$,

\[
\boxed{\max_{A\in\Omega_{2k+1}}\operatorname{per}(I-A)
=3\cdot2^{k-2}.}
\tag{1.4}
\]

The matrix $A_\star$ in (1.2) attains equality. Moreover, equality holds for
$A\in\Omega_{2k+1}$ if and only if

\[
A=PA_\star P^{\mathsf T}
\tag{1.5}
\]

for a permutation matrix $P$. Thus the number of distinct maximizers is
$\binom{2k+1}{3}(2k-3)!!$, with the convention $(-1)!!=1$.

\end{theorem}

Here is the proof in outline. Independently for each $i\in[n]$, choose
$F(i)=j$ with probability $a_{ij}$. These choices define a random map
$F:[n]\to[n]$. Let $m_c$ be the total product mass of maps whose functional
digraphs have only even cycles and exactly $c$ cycles, and put
$M=\sum_c m_c$. The Kim--Roush expansion and a layer estimate give

\[
\operatorname{per}(I-A)=\sum_{c=1}^k2^c m_c
\le 2^{k-1}(M+m_k).
\tag{1.6}
\]

Let $Q_3$ be the mass of permutation maps with cycle type
$(3,2,\ldots,2)$. The key estimate is the sharp triangle-charge inequality

\[
m_k-Q_3\le\frac12.
\tag{1.7}
\]

Every map counted by $Q_3$ has an odd cycle, so $Q_3\le1-M$. Consequently
(1.7) implies $M+m_k\le3/2$, which inserted into (1.6) proves the upper
bound in (1.4).

The proof of (1.7) has two exact ingredients. First, averaging a matrix with
its transpose does not decrease the difference $m_k-Q_3$; the local loss is
a sum of three squares. Second, for symmetric weights, releasing the terminal
row of an ordered spanning path on three vertices yields a mass $B$ satisfying
$2m_k\le B$ and

\[
B=m_k+3Q_3+O,
\]

where $O$ is a disjoint residual-map mass. Since
$O\le1-m_k-Q_3$, inequality (1.7) follows. The factor $3$ is the number of
ways to open a directed triangle into a directed path.

Sections 2--5 develop the map expansion and prove the triangle-charge
inequality. Section 6 establishes the maximum, and Section 7 classifies
equality.

\section{Maps, masses, and matchings}

Fix $n=2k+1$ with $k\ge1$, and let $A=(a_{ij})\in\Omega_n$. Write
$[n]=\{1,\ldots,n\}$. For a map $f:[n]\to[n]$, define its product mass by

\[
\mu_A(f)=\prod_{i=1}^n a_{i,f(i)}.
\tag{2.1}
\]

The row sums of $A$ are one, and therefore

\[
\sum_{f:[n]\to[n]}\mu_A(f)=\prod_{i=1}^n\sum_{j=1}^n a_{ij}=1.
\tag{2.2}
\]

The directed graph of $f$ has the edge $i\to f(i)$ from every vertex.
Each weak component contains one directed cycle. Let $c(f)$ denote the
number of these directed cycles.

For $1\le c\le k$, let $\mathcal E_c$ be the maps all of whose directed
cycles have even length and for which $c(f)=c$. Set

\[
m_c=m_c(A)=\sum_{f\in\mathcal E_c}\mu_A(f),
\qquad
M=M(A)=\sum_{c=1}^k m_c.
\tag{2.3}
\]

The families $\mathcal E_c$ are disjoint subsets of all maps, so $0\le M\le1$.

For a finite set $U$ of even cardinality, let $\operatorname{PM}(U)$ denote
the set of perfect matchings of $U$. If
$W=(w_{ij})_{i,j\in U}$ is symmetric, write

\[
\operatorname{haf}W
=\sum_{\nu\in\operatorname{PM}(U)}
\prod_{\{r,s\}\in\nu}w_{rs}.
\tag{2.4}
\]

We use the convention that the hafnian of the empty matrix equals one. Put

\[
C(A)=A\circ A^{\mathsf T},
\qquad C(A)_{ij}=a_{ij}a_{ji},
\tag{2.5}
\]

where $\circ$ denotes the entrywise (Hadamard) product. For
$U\subseteq[n]$, let $W[U]$ denote the principal submatrix of $W$ indexed by
$U$. Whenever $[n]\setminus S$ has even cardinality, define

\[
H_S(A)=\operatorname{haf}\bigl(C(A)[[n]\setminus S]\bigr).
\tag{2.6}
\]

Thus $H_S(A)$ is the total weight of mutual-pair matchings outside $S$.

\section{The positive expansion and the top layer}

We first give a self-contained proof of the expansion from \cite{kimroush1981}.

\begin{proposition}[even-cycle expansion]
If $A$ is nonnegative and
row-stochastic, then

\[
\operatorname{per}(I-A)=
\sum_{\substack{f:[n]\to[n]\\
\text{every cycle of }f\text{ is even}}}
2^{c(f)}\mu_A(f).
\tag{3.1}
\]

In particular, for $n=2k+1$,

\[
\operatorname{per}(I-A)=\sum_{c=1}^k2^c m_c.
\tag{3.2}
\]

\end{proposition}

\begin{proof}
For $S\subseteq[n]$, let $\mathfrak S(S)$ denote the set of
permutations of $S$. Expand the permanent according to the rows in which the
term from $-A$ is selected. This gives

\[
\operatorname{per}(I-A)
=\sum_{S\subseteq[n]}(-1)^{|S|}
\sum_{\pi\in\mathfrak S(S)}\prod_{i\in S}a_{i,\pi(i)},
\tag{3.3}
\]

where the empty product is one. For a fixed pair $(S,\pi)$,
row-stochasticity gives

\[
\prod_{i\in S}a_{i,\pi(i)}
=\sum_{\substack{f:[n]\to[n]\\f|_S=\pi}}\mu_A(f).
\tag{3.4}
\]

After interchanging the finite sums, fix a map $f$. The subsets $S$ for
which $f|_S$ is a permutation are exactly the unions of directed cycles of
$f$. Hence the coefficient of $\mu_A(f)$ is

\[
\prod_{\gamma\text{ a cycle of }f}
\left(1+(-1)^{|\gamma|}\right).
\tag{3.5}
\]

This coefficient vanishes if $f$ has an odd cycle and equals $2^{c(f)}$
otherwise. Since every finite map has a cycle, and $c$ even cycles use at
least $2c$ vertices, in odd order $1\le c\le k$. This proves both formulas.
\end{proof}

The next reduction isolates the only layer requiring sharp control.

\begin{lemma}[layer reduction]
For $A\in\Omega_{2k+1}$,

\[
\operatorname{per}(I-A)\le2^{k-1}(M+m_k).
\tag{3.6}
\]

\end{lemma}

\begin{proof}
In (3.2), bound $2^c\le2^{k-1}$ for $c\le k-1$, while retaining
the coefficient $2^k$ of the top layer. Thus

\[
\operatorname{per}(I-A)
\le2^{k-1}(M-m_k)+2^k m_k
=2^{k-1}(M+m_k).
\]

For $k=1$ the first sum is empty and the displayed relation is an identity.
\end{proof}

Call a map in $\mathcal E_k$ a \textbf{top map}. Such a map has a particularly
rigid form.

\begin{lemma}[top-layer structure]
One has

\[
m_k
=\sum_{v=1}^n(1-a_{vv})H_{\{v\}}(A)
=\sum_{\substack{v,u,w\in[n]\\v,u,w\text{ distinct}}}
a_{vu}a_{uw}a_{wu}H_{\{v,u,w\}}(A).
\tag{3.7}
\]

\end{lemma}

\begin{proof}
A map in $\mathcal E_k$ has $k$ even cycles on $2k+1$ vertices.
It must therefore have $k$ transpositions covering $2k$ vertices and one
vertex $v$ not lying on a directed cycle (hence one nonperiodic vertex). The
vertex $v$ maps to some $u\ne v$, while the remaining vertices form a perfect
matching of mutual pairs. For fixed $v$ and fixed matching, the sum over the
target of $v$ is
$\sum_{u\ne v}a_{vu}=1-a_{vv}$, giving the first formula.

For the second formula, let $w$ be the mate of the chosen target $u$.
The distinguished part of the map has weight
$a_{vu}a_{uw}a_{wu}$, and the remaining mutual pairs contribute
$H_{\{v,u,w\}}(A)$. The tail vertex, its target, and the target's mate are
uniquely determined by the map.
\end{proof}

\section{Triangle charge and transpose symmetrization}

Let $\mathcal Q_3$ be the set of permutation maps whose cycle type is
$(3,2,\ldots,2)$: one directed $3$-cycle and $k-1$ transpositions. Define

\[
Q_3=Q_3(A)=\sum_{f\in\mathcal Q_3}\mu_A(f).
\tag{4.1}
\]

Fix a three-element set $T=\{i,j,\ell\}$. The vertices outside $T$ must be
paired into transpositions, whose total weight is $H_T(A)$, while the
directed $3$-cycle on $T$ has exactly two orientations. Therefore

\[
Q_3(A)=\sum_{|T|=3}
\left(a_{ij}a_{j\ell}a_{\ell i}
+a_{i\ell}a_{\ell j}a_{ji}\right)H_T(A).
\tag{4.2}
\]

Each map in $\mathcal Q_3$ contains an odd cycle, whereas all maps counted by
$M$ have only even cycles. These map families are disjoint, so (2.2) gives

\[
Q_3\le1-M.
\tag{4.3}
\]

We call $m_k-Q_3$ the \textbf{triangle charge}. To express it locally, group
(3.7) and (4.2) by their distinguished three-element set and define

\[
K_T(A)=
\sum_{\substack{v,u,w\in T\\v,u,w\text{ distinct}}}
a_{vu}a_{uw}a_{wu}
-\frac13
\sum_{\substack{v,u,w\in T\\v,u,w\text{ distinct}}}
a_{vu}a_{uw}a_{wv}.
\tag{4.4}
\]

The normalization $1/3$ removes a simple overcount; this correction is needed
for identity (4.5). In the first sum, the ordered triple $(v,u,w)$ records a
top map's nonperiodic vertex $v$, its target $u$, and the mate $w$ of $u$. In
the second sum,
$a_{vu}a_{uw}a_{wv}$ is the weight of the directed triangle
$v\to u\to w\to v$. Each directed triangle on $T$ occurs three times in
that sum, once for each cyclic choice of the first vertex $v$. Thus the
second sum, after multiplication by $1/3$, is exactly the sum of the two
orientation weights in (4.2). It follows that

\[
m_k(A)-Q_3(A)=\sum_{|T|=3}K_T(A)H_T(A).
\tag{4.5}
\]

Although writing $T=\{i,j,\ell\}$ in (4.2) requires an arbitrary ordering,
the sum of the two orientation weights in (4.2), the quantity $K_T(A)$, and
identity (4.5) depend only on the set $T$.

For nonnegative $x,y,z$, put

\[
\begin{aligned}
\Phi(x,y,z)
&=xy(x+y)+yz(y+z)+zx(z+x)-2xyz\\
&=(x+y)(y+z)(z+x)-4xyz\\
&=x(y-z)^2+y(z-x)^2+z(x-y)^2+4xyz.
\end{aligned}
\tag{4.6}
\]

In particular, $\Phi(x,y,z)\ge0$.

\begin{lemma}[local transpose identity]
Let
$X=(A+A^{\mathsf T})/2$. Orient $T=\{i,j,\ell\}$ cyclically as
$i\to j\to\ell\to i$, and write

\[
\begin{array}{lll}
a_{ij}=x+p,&a_{ji}=x-p,&a_{j\ell}=y+q,\\
a_{\ell j}=y-q,&a_{\ell i}=z+r,&a_{i\ell}=z-r.
\end{array}
\tag{4.7}
\]

Then

\[
\frac{K_T(A)+K_T(A^{\mathsf T})}{2}
=\Phi(x,y,z)-D_T,
\tag{4.8}
\]

where

\[
D_T=y(p+r)^2+z(p+q)^2+x(q+r)^2\ge0.
\tag{4.9}
\]

\end{lemma}

\begin{proof}
Substitute the six expressions in (4.7) into (4.4), add the
expression obtained after changing $(p,q,r)$ to $(-p,-q,-r)$, and collect
the even terms. The terms independent of $(p,q,r)$ give (4.6), and the
remaining even terms are the negative of (4.9).
\end{proof}

For the variables attached to $T$ in (4.7), abbreviate
$\Phi_T=\Phi(x,y,z)$.

\begin{theorem}[transpose symmetrization]
For every
$A\in\Omega_{2k+1}$ and $X=(A+A^{\mathsf T})/2$,

\[
m_k(A)-Q_3(A)\le m_k(X)-Q_3(X).
\tag{4.10}
\]

\end{theorem}

\begin{proof}
Here the column-sum condition enters the upper-bound argument:
because $A$ is doubly stochastic, $A^{\mathsf T}$ is row-stochastic and
$X$ is symmetric and doubly stochastic. The quantities $m_k$ and $Q_3$ are
transpose invariant. For $m_k$, apply the first formula in (3.7) to both
$A$ and $A^{\mathsf T}$ and use $C(A)=C(A^{\mathsf T})$; for $Q_3$,
transposition exchanges the two orientations in (4.2). Moreover
$H_T(A)=H_T(A^{\mathsf T})$. We may therefore replace $K_T(A)$ in (4.5)
by its transpose average.

For an edge $\{r,s\}$ outside $T$, write

\[
x_{rs}=\frac{a_{rs}+a_{sr}}2,
\qquad
\kappa_{rs}=\frac{a_{rs}-a_{sr}}2.
\]

Then

\[
a_{rs}a_{sr}=x_{rs}^2-\kappa_{rs}^2,
\qquad
0\le x_{rs}^2-\kappa_{rs}^2\le x_{rs}^2.
\tag{4.11}
\]

Expanding the hafnian in (4.5) over perfect matchings and using Lemma 4.1
expresses the charge as a sum of terms

\[
(\Phi_T-D_T)
\prod_{\{r,s\}\in\nu}(x_{rs}^2-\kappa_{rs}^2).
\tag{4.12}
\]

The following sign distinction is essential. If $\Phi_T-D_T\le0$, the term
in (4.12) is nonpositive and hence at most
$\Phi_T\prod_{\{r,s\}\in\nu}x_{rs}^2$. If $\Phi_T-D_T>0$, first discard
$D_T$ and then use (4.11), obtaining the same upper bound. Summing over
triples and outside perfect matchings gives

\[
m_k(A)-Q_3(A)
\le\sum_{|T|=3}\Phi_T H_T(X)
=m_k(X)-Q_3(X),
\]

where the last equality is (4.5) specialized to the symmetric matrix $X$.
\end{proof}

\begin{remark}[why the transpose average is necessary]
The unaveraged
local inequality suggested by (4.8) is false. For example, on a cyclically
oriented triple take

\[
(a_{12},a_{21},a_{13},a_{31},a_{23},a_{32})
=\left(0,\frac12,\frac12,\frac12,\frac12,0\right).
\tag{4.13}
\]

The ungrouped comparison fails by exactly $1/8$. These six entries have row
and column subsums at most one and extend to a doubly stochastic matrix.
The odd-in-skew remainder cancels only after pairing $A$ with
$A^{\mathsf T}$. Thus the transpose invariance used in Theorem 4.2 is a
structural part of the proof, not a cosmetic symmetrization.

\end{remark}

\section{The symmetric charge inequality}

We now prove the sharp estimate for a symmetric doubly stochastic matrix
$X=(x_{ij})$ of order $2k+1$. All weights below are entries of $X$.

For distinct vertices $v,u,w$ and a perfect matching $\nu$ of the other
$2k-2$ vertices, define the released-path weight

\[
h(v,u,w;\nu)
=x_{vu}x_{uw}\prod_{\{a,b\}\in\nu}x_{ab}^2.
\tag{5.1}
\]

It fixes the choices

\[
v\longmapsto u,
\qquad u\longmapsto w,
\qquad a\longleftrightarrow b\quad(\{a,b\}\in\nu),
\tag{5.2}
\]

and releases the terminal row $w$. Indeed, summing $w\mapsto t$ over all
$t$ multiplies (5.1) by $\sum_t x_{wt}=1$. Let $B$ be the total
released-path mass:

\[
B=\sum_{\substack{v,u,w\text{ distinct}\\
\nu\in\operatorname{PM}([n]\setminus\{v,u,w\})}}
h(v,u,w;\nu).
\tag{5.3}
\]

\begin{lemma}[released mass dominates the top layer]
One has

\[
2m_k(X)\le B.
\tag{5.4}
\]

\end{lemma}

\begin{proof}
Pair the paths $(v,u,w)$ and $(w,u,v)$ while fixing the center $u$
and the outside matching $\nu$. Put

\[
a=x_{vu},\qquad b=x_{uw},\qquad
H=\prod_{\{r,s\}\in\nu}x_{rs}^2.
\]

Their combined released mass is $2abH$. Closing their terminal rows back to
$u$ produces two top maps of total mass

\[
ab^2H+a^2bH=ab(a+b)H.
\]

Since the center row is stochastic,

\[
a+b\le\sum_{z\ne u}x_{uz}=1-x_{uu}\le1.
\]

Thus twice the mass of the two top completions is at most their released
mass. Every top map has a unique tail vertex $v$, tail target $u$, mate $w$
of $u$, and outside matching $\nu$, so summing proves (5.4). Equivalently,
the slack is the manifestly nonnegative sum

\[
B-2m_k
=\sum_{u,\{v,w\},\nu}
2x_{vu}x_{uw}(1-x_{vu}-x_{uw})
\prod_{\{a,b\}\in\nu}x_{ab}^2.
\tag{5.5}
\]

Here the sum is over centers $u$, unordered endpoint pairs $\{v,w\}$
disjoint from $u$, and matchings
$\nu\in\operatorname{PM}([n]\setminus\{u,v,w\})$. The right-hand side is
nonnegative.
\end{proof}

We next classify the maps obtained when the released row is completed.

\begin{lemma}[completion identity]
There is a nonnegative map mass $O$,
disjoint from the masses $m_k$ and $Q_3$, such that

\[
B=m_k+3Q_3+O
\tag{5.6}
\]

and

\[
O\le1-m_k-Q_3.
\tag{5.7}
\]

\end{lemma}

\begin{proof}
In a released configuration (5.2), classify the completion by the
choice $w\mapsto t$.

If $t=u$, the terminal row closes the mutual pair $u\leftrightarrow w$ and
produces a top map. Its representation is unique, so every top map occurs
once.

If $t=v$, the path closes into the directed triangle
$v\to u\to w\to v$. Every map counted by $Q_3$ occurs three times, once for
each cyclic choice of where its directed triangle is opened.

Finally, suppose $t\notin\{u,v\}$, allowing either $t=w$ or $t$ in an
outside pair. Call the resulting map residual. Its mutual pairs are exactly
the pairs of $\nu$: none of $\{v,u\}$, $\{u,w\}$, or $\{w,t\}$ is mutual.
After removing those pairs, the three remaining vertices contain the unique
directed path $v\to u\to w$; the vertex $v$ is its unique vertex of
in-degree zero, and $w$ either loops or maps out to a removed pair. Hence
$v,u,w$, and $\nu$ are uniquely reconstructed from the residual map. Every
residual map therefore occurs once.

These three classes are disjoint: top maps have exactly one nonperiodic
vertex, $Q_3$ maps are permutations with a directed odd cycle, and residual
maps have at least two nonperiodic vertices. If $O$ is the total product
mass of the residual maps, the multiplicities are exactly $1,3,1$, proving
(5.6). All three classes lie inside the map space of total mass one in
(2.2), which proves (5.7).
\end{proof}

\begin{theorem}[symmetric triangle charge]
For every symmetric
$X\in\Omega_{2k+1}$,

\[
m_k(X)-Q_3(X)\le\frac12.
\tag{5.8}
\]

\end{theorem}

\begin{proof}
Lemmas 5.1 and 5.2 give

\[
2m_k\le B=m_k+3Q_3+O
\le m_k+3Q_3+(1-m_k-Q_3)=1+2Q_3.
\]

Rearranging proves (5.8).
\end{proof}

Combining Theorems 4.2 and 5.3 yields the promised universal estimate.

\begin{corollary}[triangle-charge inequality]
For every
$A\in\Omega_{2k+1}$,

\[
\boxed{m_k(A)-Q_3(A)\le\frac12.}
\tag{5.9}
\]

\end{corollary}

\section{Proof of the main theorem}

We now assemble the short final implication. By (4.3) and (5.9),

\[
M+m_k\le1-Q_3+m_k\le\frac32.
\tag{6.1}
\]

Lemma 3.2 therefore gives

\[
\operatorname{per}(I-A)
\le2^{k-1}(M+m_k)
\le2^{k-1}\cdot\frac32
=3\cdot2^{k-2}.
\tag{6.2}
\]

It remains to check sharpness. For the matrix $A_\star$ in (1.2), block
multiplicativity of the permanent gives

\[
\begin{aligned}
\operatorname{per}(I-A_\star)
&=\operatorname{per}\!\left(I_3-\frac12(J_3-I_3)\right)
\operatorname{per}(I_2-P_2)^{k-1}\\
&=\frac32\,2^{k-1}
=3\cdot2^{k-2}.
\end{aligned}
\tag{6.3}
\]

Thus equality is attained, proving the maximum-value part of Theorem 1.1.
The argument includes
$k=1$: empty outside matchings have weight one, $Q_3$ is then the mass of
the two oriented $3$-cycles, and (6.3) gives $3/2$. \hfill\(\square\)

At $A_\star$, the map masses are

\[
(M,m_k,Q_3)=\left(\frac34,\frac34,\frac14\right).
\tag{6.4}
\]

Consequently both the triangle-charge constant $1/2$ and the compensation
constant $3/2$ in (6.1) are sharp.

\section{Classification of equality}

We now prove the equality assertion in Theorem 1.1. Suppose that
$A\in\Omega_{2k+1}$ attains the maximum. The proof above gives the monotone
chain

\[
\begin{aligned}
\operatorname{per}(I-A)
&\le 2^{k-1}(M+m_k)\\
&\le 2^{k-1}\left(M+Q_3+\frac12\right)\\
&\le 2^{k-1}\cdot\frac32.
\end{aligned}
\tag{7.1}
\]

Equality at the two ends forces equality at every step, and in particular

\[
Q_3=1-M.
\tag{7.2}
\]

The maps counted by $M$ and $Q_3$ are disjoint subsets of the map space of
total mass one in (2.2). Since every map mass is nonnegative, (7.2) has the
following support consequence:

\begin{quote}
Every map $f$ satisfying $a_{i,f(i)}>0$ for every $i$ either has only even
cycles, or is a permutation of cycle type $(3,2,\ldots,2)$.
\end{quote}

We call this property the support dichotomy. Let $S$ be the support digraph
of $A$, with a directed edge $i\to j$ exactly when $a_{ij}>0$. Any partial
assignment of supported outgoing edges extends to a full positive-mass map:
in each unused row, simply choose any positive entry. We use this observation
repeatedly.

First, $S$ has no loop. Indeed, a map extending a supported loop has an odd
$1$-cycle, so it does not have only even cycles, and its fixed point prevents
it from having type $(3,2,\ldots,2)$. The same extension argument shows that
$S$ has no directed odd cycle of length at least five.

By the Birkhoff--von Neumann theorem, $S$ contains the edges of a
permutation. This permutation has no fixed point. Its cycle lengths sum to
the odd number $2k+1$, so one of them is odd; by the preceding paragraph it
must have length three. Thus $S$ contains a directed triangle $\gamma$ on
a vertex triple $T$.

No supported edge enters $T$ from outside. Otherwise, extend $\gamma$
together with an edge $z\to t$, where $z\notin T$ and $t\in T$. The
resulting map has the odd cycle $\gamma$, but it is not a permutation because
$t$ has two distinct preimages. This contradicts the support dichotomy. Each
column in $T$ therefore receives its full unit mass from rows in $T$. Those
three column sums exhaust the total mass of the three rows in $T$, so no
supported edge leaves $T$ either. After simultaneous relabelling,

\[
A=A_1\oplus A_2,
\qquad A_1\in\Omega_3,\quad A_2\in\Omega_{2k-2},
\tag{7.3}
\]

and both blocks have zero diagonal.

Choose arbitrarily one supported outgoing edge from each row of $A_2$ and
combine these choices with $\gamma$. The resulting map contains an odd
triangle, so the support dichotomy forces it to be a
$(3,2,\ldots,2)$-permutation. Consequently every such row choice in the
complement is a fixed-point-free involution.

If a complementary row $z$ had two supported targets $t_1\ne t_2$, choose
such an involution $g$ with $g(z)=t_1$. Then $g(t_1)=z$, and
$g(t_2)\ne z$. Changing only the choice in row $z$ from $t_1$ to $t_2$
would produce another supported row choice, but its edge $z\to t_2$ would
not be reciprocated. This contradicts the preceding universal involution
property. Hence every row of $A_2$ has exactly one positive entry. After
relabelling the complementary vertices, row stochasticity and the support
dichotomy now give

\[
A_2=P_2^{\oplus(k-1)}.
\tag{7.4}
\]

Finally, $A_1$ is a zero-diagonal $3\times3$ doubly stochastic matrix. Let
$C$ be a cyclic permutation matrix. The only fixed-point-free permutations
of three vertices are $C$ and $C^{\mathsf T}$, so Birkhoff--von Neumann gives

\[
A_1=xC+(1-x)C^{\mathsf T},
\qquad 0\le x\le1,
\tag{7.5}
\]

Direct expansion yields

\[
\operatorname{per}(I_3-A_1)
=1-x^3-(1-x)^3+3x(1-x)
=6x(1-x).
\tag{7.6}
\]

Using (7.4) and block multiplicativity,

\[
\operatorname{per}(I-A)=6x(1-x)2^{k-1}.
\tag{7.7}
\]

Equality with $3\cdot2^{k-2}$ forces $x(1-x)=1/4$, hence $x=1/2$.
Thus $A_1=\tfrac12(J_3-I_3)$, proving (1.5). The converse is (6.3).

It remains to count the distinct matrices. Choose the three vertices of the
half-weight triangle in $\binom{2k+1}{3}$ ways, then choose a perfect
matching of the other $2k-2$ vertices in $(2k-3)!!$ ways. The triangle is
symmetric, so there is no orientation factor, and distinct choices produce
distinct matrices. With the convention $(-1)!!=1$, this also covers
$k=1$. This completes the proof of Theorem 1.1. \hfill\(\square\)

\section*{Acknowledgments}

The proof of this conjecture was carried out by GPT-5.6-sol and Claude Fable
5, under the guidance of the author. The author has reviewed the resulting
proof arguments. Responsibility for the final text rests with the author.

\end{document}